\documentclass[11pt,a4paper]{article}

\usepackage{amsmath}          
\usepackage{amssymb}  
\usepackage{amsmath,amsthm} 
\usepackage{enumerate}
\usepackage{stmaryrd}
\usepackage{url} 
\usepackage{a4wide}

\usepackage[round]{natbib}       

\usepackage{times}

\newtheorem{theorem}{Theorem}[section]   
\newtheorem{lemma}[theorem]{Lemma}
\newtheorem{definition}[theorem]{Definition}

\numberwithin{equation}{section}         

\def\mb#1{\mathbf{#1}}

\def\L{{\mb L}}

\def\so{\mathsf{1}}
\def\sz{\mathsf{0}}

\renewcommand{\bot}{\mathsf{F}}
\renewcommand{\top}{\mathsf{T}}

\def\c#1{\mathcal{#1}}

\def\LL{\mathcal{L}}
\def\M{\mathcal{M}}         
\def\S{\mathcal{S}}

\def\e{\varepsilon}
\def\ph{\varphi}

\def\abm#1{|#1|^\M}

\def\bcap{\mathop{\textstyle\bigcap}}
\def\bcup{\mathop{\textstyle\bigcup}}
\def\bscap{\mathop{\textstyle\bigsqcap}}
\def\bscup{\mathop{\textstyle\bigsqcup}}

\def\cd{\cdot}
\def\cov{\vartriangleright}
\def\covby{\vartriangleleft}

\def\di#1{\langle #1\rangle}

 \def\emp{\emptyset}

 \def\fu{\mathbin{\otimes}}    
 
 \def\jc{j}
 \def\join{\bscup}
 
 \def\meet{\bscap}

\def\pre{\preccurlyeq}      

\def\sub{\subseteq}

\def\sqleq{\mathrel{\sqsubseteq}}

\def\suc{\succcurlyeq}
 
 \def\li{\rightarrow_{l}}
\def\ri{\rightarrow_{r}}

\def\lmin{\mathop{-_{l}}}
\def\rmin{\mathop{-_{r}}}
\def\imin{\mathop{-_{i}}}

\def\minus{\mathop{-}}

\def\lneg{\neg_l}
\def\rneg{\neg_r}

\def\lr{\Rightarrow_{l}}
\def\rr{\Rightarrow_{r}}
\def\ir{\Rightarrow_{i}}

\def\sh{\mathop{!}}
\def\qu{\mathop{?}}
\def\qus{{\dot{\qu}}}

\def\up{{\uparrow}}

\def\<{\langle}
\def\>{\rangle}

\def\iff{\quad \text{iff} \quad}

\def\Prop{\mathit{Prop}}
\def\Up{\mathit{Up}}

\begin{document}

\title{Cover Systems for the Modalities of Linear Logic}
\author{Robert Goldblatt\thanks{School of Mathematics and Statistics, Victoria University of Wellington, New Zealand.
{\tt sms.vuw.ac.nz/\~{}rob}}}
\maketitle


\abstract{
Ono's modal FL-algebras are  models of an extension of Full Lambek logic that has the modalities !  and ? of linear logic. Here we define a notion of modal FL-cover system that combines aspects of Beth-Kripke-Joyal semantics with Girard's interpretation of the ! modality, and has structured subsets that interpret propositions. We
show that any modal FL-algebra can be represented as an algebra of propositions of some modal FL-cover system.}


\section{Introduction}

Hiroakira Ono pioneered the development of  Kripke-style semantic interpretations of substructural logics, beginning with work on logics that lack the contraction rule \citep{ono:logi85,ono:sema85}.
His fundamental article \citep{ono:sema93} then gave a detailed analysis, involving both algebraic and Kripke-type models, for extensions of \emph{Full Lambek logic} (FL), described roughly as the Gentzen sequent calculus obtained from that for intuitionistic logic by deleting all the structural rules.

Included in this analysis were connectives $\sh$ and $\qu$ corresponding to the \emph{storage} and \emph{consumption} modalities of the linear logic of \citet{gira:line87}. The Kripke models for these in \citep{ono:sema93} were certain relational structures based on semilattice-ordered monoids that carried binary relations to interpret $\sh$ and $\qu$. 

In the present paper we give an alternative modelling of this \emph{modal FL logic} using \emph{cover systems} that are motivated by the topological ideas underlying the Kripke-Joyal semantics for intuitionistic logic in topoi. A cover system assigns to each point certain sets of points called ``covers'' in a way that is formally similar to the neighbourhood semantics of modal logics. Covers are used to give non-classical interpretations of disjunction and existential quantification, and in that sense are also reminiscent of Beth's intuitionistic semantics. The present author has previously developed cover system semantics for the (non-distributive) non-modal FL-logic, as well as for relevant logics and intuitionistic modal logics.\footnote{Bibliographical references for these are given at the end of Section \ref{sec7}.}

Our treatment of the storage modality $\sh$ abstracts from that of the phase space semantics of \cite{gira:line95}, which  is based on commutative monoids with a certain closure operator on its subsets. Propositions are interpreted there as closed subsets, called \emph{facts}, and $\sh X$ is defined to be the least fact including $X\cap I$, where $I$ is the set of all monoid idempotents  that belong to $1$, the least fact containing the monoid identity $\e$. Here, as well as abandoning the commutativity in order to model FL-logic in general, we allow $I$ to be a submonoid of this set of idempotents that forms a cover of $\e$. We also require $I$ to be \emph{central}, i.e.\ its elements commute with all elements. Our models, which are called \emph{modal FL-cover systems},  also have a quasiordering  that is used to interpret the consumption modality $\qu$ by the Kripkean existential clause for a classical $\Diamond$-style modality. But it should be appreciated that in this non-commutative and non-distributive setting, a modality with this existential interpretation need not distribute over disjunction in the way that a classical $\Diamond$ does.

Propositions for us are ``localised up-sets'' that are defined by the cover system structure (see Section \ref{sec3}). We show that the set of propositions of a modal FL-cover system satisfies Ono's axioms for a \emph{modal FL-algebra}. Any order-complete modal FL-algebra is shown to be isomorphic to the algebra of propositions of some modal FL-cover system, while an arbitrary modal FL-algebra can be embedded into the algebra of propositions of a modal FL-cover system by an embedding that preserves any existing joins and meets. We also show that in any FL-algebra with a storage modality $\sh$, the term function $-\sh - a$ defines a modality satisfying the axioms for $\qu$ so gives rise to a modal FL-algebra. Here the two occurrences of $-$ can stand separately for either of the two negation operations that exist in any FL-algebra.

\section{Modalities on Residuated Lattices}      \label{sec2}

A  \emph{residuated partially ordered monoid} (or residuated \emph{pomonoid}), can be defined as an algebra of the form
$$
\L=(L,\sqleq,\fu,1,\lr,\rr),
$$
such that:
\begin{itemize}
\item
$\sqleq$ is a partial ordering on the set $L$.

\item
$(L,\fu,1)$ is a \emph{monoid}, i.e.\  $\fu$ is an associative binary operation (called \emph{fusion}) on $L$, with identity element 1, that is $\sqleq$-monotone in each argument:  $b\sqleq c$ implies $a\fu b\sqleq a\fu c$ and $b\fu a\sqleq c\fu a$. 
\item
 $\lr$ and $\rr$ are binary operations on $L$,\footnote{Notation:  in the literature on residuation,  $a\lr b$ is often written as $b/a$, and $a\rr b$ as $a\backslash b$.}
 called the \emph{left and right  residuals} of $\fu$, satisfying the \emph{residuation law}
 \begin{equation*} 
a\sqleq b\lr c  \iff  a\fu b\sqleq c  \iff   b\sqleq a\rr c.
\end{equation*}
\end{itemize}

A \emph{residuated lattice} is a residuated pomonoid that is a lattice under $\sqleq$, with binary join operation $\sqcup$ and meet operation $\sqcap$. We also write $\join$ and $\meet$ for the join and meet operations on subsets of $L$ when these operations are defined.

 \citet{gala:resi07} give an extensive treatment of the theory of residuated lattices and its application to substructural 
logic.  They define an \emph{FL-algebra} (Full Lambek algebra) to be a residuated lattice with an additional distinguished element 0.
We will mainly deal  with  lattices that are \emph{bounded}, i.e.\ have a
greatest element $\top$ and least element $\bot$. For this it suffices that there be a least element $\bot$, for then there is  a greatest element $\bot\lr\bot=\bot\rr\bot$. 

\begin{definition} \label{def1}
A \emph{storage modality} on a residuated lattice $\L$ is  a unary operation $\sh$ on $L$ such that
\begin{enumerate}[(s1)] 
\item
$\sh a\sqleq a$.
\item
$\sh a\sqleq\sh \sh a$.
\item
$\sh 1=1$.
\item
$\sh(a\sqcap b)=\sh a\fu\sh b$.
\item
$\sh a\fu b=b\fu\sh a$.
\end{enumerate}
\end{definition}

\begin{lemma}  \label{lem1}
Any storage modality satisfies the following.
\begin{enumerate}[\rm(1)]
\item
$\sh a\sqleq 1$.
\item
$\sh$ is monotone, i.e.\ $a\sqleq b$ implies $\sh a\sqleq \sh b$.  
\item
$\sh a=\sh a\fu\sh a$.
\item
$\sh a\fu\sh b=\sh (\sh a\fu\sh b)\sqleq\sh ( a\fu b)$.
\item
$\sh(a\ir b)\sqleq\sh a\ir\sh b$ \quad for $i=l,r$.
\item
$\sh\top=1$ \enspace( if\/ $\top$ exists).

\end{enumerate}
\end{lemma}

\begin{proof}
\begin{enumerate}[\rm(1)]
\item
Using (s3), (s4) and then (s1), with lattice properties, we observe
$$
\sh a=\sh a\fu 1=\sh a\fu \sh 1=\sh(a\sqcap 1)\sqleq a\sqcap 1\sqleq 1.
$$
\item
If $a\sqleq b$, then $\sh a=\sh(a\sqcap b)=\sh a\fu\sh b$ by (s4). But by (1) and monotonicity of $\fu$,
we get
$\sh a\fu\sh b\sqleq1\fu\sh b=\sh b$.
\item
Put $a=b$ in (s4).
\item
By (s1) and (s2), $\sh(a\sqcap b)=\sh\sh(a\sqcap b)$. This becomes
$\sh a\fu\sh b=\sh (\sh a\fu\sh b)$  by (s4). But $\sh a\fu\sh b\sqleq a\fu b$ by (s1) and monotonicity of $\fu$.
So $\sh (\sh a\fu\sh b)\sqleq\sh ( a\fu b)$ by (2).
\item
$(a\lr b)\fu a\sqleq b$, so $\sh((a\lr b)\fu a)\sqleq \sh b$ by (2). Then using (4),
$\sh(a\lr b)\fu \sh a\sqleq \sh((a\lr b)\fu a)\sqleq \sh b$, hence 
$\sh(a\lr b)\sqleq\sh a\lr\sh b$.

The case of $\rr$ is similar, using $a\fu(a\rr b)\sqleq b$.
\item
$\sh\top\sqleq 1$ by (1). But $1=\sh 1\sqleq\sh\top$ by (s3) and (2).
\qed
\end{enumerate}
\end{proof}

The axioms (s1)--(s5) form part of  Ono's definition of a modal FL-algebra, which we come to shortly. Note that (s5) was not used at all in Lemma \ref{lem1}.

\citet{troe:lect92} deals with \emph{IL-algebras} (\emph{intuitionistic linear} algebras), which are essentially bounded residuated lattices in which $\fu$ is commutative, and hence $\lr$ and $\rr$ are identical. He defines an \emph{ILS-algebra} (intuitionistic linear algebra with \emph{storage}), to be  an IL-algebra with a unary operator $\sh$ having $\sh a\sqleq a$; $\sh a\sqleq b$ only if $\sh a\sqleq \sh b$; $\sh\top=1$; and $\sh(a\sqcap b)=\sh a\fu\sh b$. These conditions together are equivalent to (s1)--(s4) in any IL-algebra, which automatically satisfies (s5) because it has commutative $\fu$.

Another equivalent definition of ILS-algebra is given by \citet[Def.~3.5]{buca:moda94}.  It has $\sh\top\sqleq 1$ and $1\sqleq\sh 1$ in place of (s3), and $\sh a\sqleq\sh a\fu\sh a$ and $\sh a\fu\sh b\sqleq\sh (\sh a\fu\sh b)$  in place of (s4).

The following notion was introduced in \citep[Definition 6.1]{ono:sema93}.
\begin{definition} \label{def2}
A \emph{modal FL-algebra} is a bounded FL-algebra with a storage modality $\sh$ and an additional unary operation $\qu$ satisfying
\begin{enumerate}
\item[(c1)]
$\sh(a\ir b)\sqleq\qu a\ir\qu b$ \quad for $i=l,r$.
\item[(c2)]
$ a\sqleq \qu a$.
\item[(c3)]
$\qu\qu a\sqleq\qu  a$
\item[(c4)]
$\qu 0\sqleq 0$.
\item[(c5)]
$ 0\sqleq \qu a$.
\end{enumerate} 
\end{definition}

\begin{lemma}  \label{lem2}
In any modal FL-algebra, $\qu$ is a monotone operation satisfying 
$$
\emph{(c6)}\ \ \sh a\fu\qu b\sqleq\qu(a\fu b) \quad\text{and}\quad  \emph{(c7)}\ \ \qu a\fu\sh b\sqleq\qu(a\fu b).
$$
\end{lemma}

\begin{proof}
We have $(a\lr b)\fu a\sqleq b$. Now let  $a\sqleq b$. Then $1\fu a\sqleq b$, so$1\sqleq a\lr b$ and hence
by (s3) and monotonicity of $\sh$,
$$
1=\sh 1\sqleq\sh(a\lr b)\sqleq  \qu a\lr\qu b,
$$
with the last inequality given by (c1) with $i=l$. Thus $\qu a=1\fu\qu a\sqleq\qu b$, establishing that $\qu$ is monotone.

For (c6), since $a\sqleq b\lr a\fu b$, we get $\sh a\sqleq \sh(b\lr a\fu b))\sqleq \qu b\lr\qu(a\fu b)$ using (c1) for $i=l$.
Hence $\sh a\fu\qu b\sqleq\qu(a\fu b)$. The proof of (c7) is similar, using $b\sqleq a\rr a\fu b$ and (c1) for $i=r$.
\qed
\end{proof}
The conditions on $\qu$ in this Lemma are equivalent to (c1) in any FL-algebra. In fact more strongly:

\begin{lemma}
Let\/ $\L$ be a residuated pomonoid with a monotone operation $\sh$ having $\sh 1=1$. Then a unary operation $\qu$ on $\L$ satisfies (c1) if, and only if it is monotone and satisfies (c6) and (c7).
\end{lemma}
\begin{proof}
The only-if part is shown by the proof of the last Lemma. Conversely, assume $\qu$ is monotone and satisfies (c6) and (c7).
By (c6) and then $\qu$-monotonicity,
$$
\sh(a\lr b)\fu\qu a\sqleq\qu((a\lr b)\fu a)\sqleq \qu b
$$
from which residuation gives (c1) for $i=l$. The case of $i=r$ is similar, using (c7).
\qed
\end{proof}

Any modal FL-algebra $\L$ can be embedded into an order-complete modal FL-algebra by an embedding that preserves any joins and meets that exist in $\L$. This was shown in Section 4 of \citep{ono:sema93}, by an extension of the MacNeille completion construction. We make use of the result below in representing modal FL-algebras over cover systems.

\section{Cover Systems}  \label{sec3}

FL-algebras will be represented as algebras of subsets of structures of the form $\S=(S,\pre,\covby,\dots )$, in which $\pre$ is a preorder (i.e.\ reflexive  transitive relation) on $S$, and $\covby$ is a binary relation from $S$ to its powerset $\c PS$.
We sometimes write $y\suc x$ when $x\pre y$, and say that $y$ \emph{refines} $x$.
 When $x\covby C$, where $x\in S$ and $C\sub S$, we say that $x$ \emph{is covered by} $C$, and write this also as $C\cov x$, saying that $C$ \emph{covers} $x$ or that $C$ is an \emph{$x$-cover}. 

An \emph{up-set} is a subset  $X$ of $S$ that is closed upwardly under refinement: $y\suc x\in X$ implies $y\in X$.   For an arbitrary $X\sub S$,  
$$
\up X=\{y\in S: (\exists x\in X)\, x\pre y\}
$$
is the smallest up-set including $X$. For $x\in S$, we write $\up x$ for $\up\{x\}=\{y:x\pre y\}$, 
 the smallest up-set containing $x$. 
 
 The collection  $\Up(\c S)$ of all up-sets of $\c S$ is a complete poset under the partial order $\sub$ of set inclusion, with the join $\bscup\c X$ and meet $\bscap\c X$ of any collection $\c X$ of up-sets being its set union  
$\bcup\c X$ and intersection $\bcap\c X$ respectively, while  $\bot=\emp$ and $\top=S$.

A subset $Y$ of $S$ \emph{refines} a subset $X$ if $Y\sub\up X$, i.e.\ if every member of $Y$ refines some member of $X$.
We call $\S$  a \emph{cover system} if it satisfies the following axioms, for all $x\in S$:

\begin{itemize}
\item
\emph {Existence}:   there exists an $x$-cover $C\sub\up x$;
\item 
\emph{Transitivity}:
if $x\covby C$ and for all $y\in C$, $y\covby C_y$, then $x\covby \bigcup_{y\in C }C_y$.
\item
\emph{Refinement}:
if $x\pre y$, then every $x$-cover  can be refined to a $y$-cover, i.e.\ if $C\cov x$, then there exists a $C'\cov y$ with $C'\sub\up C$.
\end{itemize}
For each subset  $X$ of $S$, define
\begin{equation}                                        \label{eq:jcov}
\jc X=\{x\in S: \exists C\,(x\covby C\sub X)\}.
\end{equation}
A property is thought of as being \emph{locally true} of $x$ if $x$ is covered by a set of members that have this property, i.e.\ if there is some $C$ such that $x\covby C$ and each member of $C$ has the property.
In this sense, $x$ belongs to $ \jc X$ just when the property of \emph{being a member of} $X$ is locally true of $x$. So  
$\jc X$ can be thought of as the collection of ``local members'' of $X$. 
$X$ is called \emph{localised} if $\jc X\sub X$, i.e.\ if every local member of $X$ is an actual member of $X$. 

It was shown in  \cite[Theorem 5]{gold:krip06} and  \cite[Lemma 3.3]{gold:cove11} that  in any cover system, the function $j$ defined by \eqref{eq:jcov} is a closure operator on the complete poset $(\Up(\S),\sub)$ of up-sets, i.e.\ $j$ is monotonic and has $X\sub jX=j(jX)$.

A \emph{proposition} in a cover system is an up-set $X$ that is localised, i.e.\  $jX\sub X$, hence $jX=X$.   In general, a set $X$ is a proposition iff $X=\up X=jX$.   $j\up X$ is the smallest proposition that includes an arbitrary $X$, and $j\up x$ is the smallest proposition containing the element $x$. The smallest proposition including an up-set $X$ is just $jX$, so in fact $j$ maps $\Up(\c S)$ onto the set $\Prop(\c S)$  of all localised up-sets of a cover system $\S$. Indeed,  $\Prop(\c S)$ is precisely the set of fixed points of this map.

Requiring propositions to be localised amounts to making truth a property of \emph{local character}, i.e. it holds whenever it does so locally. For further discussion of this see  \cite{gold:cove11}, or \cite[Section 6.3]{gold:quan11} where an information-theoretic interpretation of the cover relation $\covby$ is also given.

Our cover systems have some formal similarities with the notion of a pretopology of  \citet{samb:intu89}, but there are some basic differences, including the presence here  of the preorder $\pre$,  and the absence  of Sambin's reflexivity condition that $x\covby C$ whenever $x\in C$. Our systems are motivated by the topological ideas underlying the Beth-Kripke-Joyal semantics for  logic in sheaf categories \citep{macl:shea92}, and  relate more to the \emph{cover schemes} on preordered sets of \citep{bell:cove06}. 

\citet[p.~72]{drag:math88} gave a method of constructing closure operators over preordered sets  that  is  motivated by the features of Beth's models. He defined an operation $\mathbf{D}$ on \emph{down}-sets of a preorder by taking a function $Q$ assigning to each $x\in S$ a collection $Q(x)$ of subsets of $S$, and putting
$
\mathbf{D}Y=\{x\in S:  \forall C\in Q(x), C \cap Y\ne\emptyset\}.
$
He gave conditions on $Q$  ensuring that \textbf{D} is a closure operator, and interpreted $C\in Q(x)$ by saying that `$C$ is a path starting from the moment $x$'. Now if we define $x\covby C$ to mean $C\in Q(x)$, then for any up-set $X$ it follows that  $S\setminus X$ is a down-set and $j_\covby X=S\setminus(\mathbf{D}(S\setminus X)$, so  in this sense Dr\'{a}galin's approach is dual to that of cover systems.  \cite{bezh:loca16} give  a detailed  discussion of the relationship between these
approaches.

Every topological space has the cover system in which $S$ is the set of open subsets of the space, with $x\pre y$ iff $x\supseteq y$ and $x\covby C$ iff $x=\bigcup C$. This system has the property that every $x$-cover is included in $\up x$, as do  the cover schemes of \citep{bell:cove06}. But this property makes  $\Prop(\c S)$ into a distributive lattice. Indeed even the weaker constraint that every $x$-cover \emph{can be refined to} an $x$-cover included in $\up x$ is enough to force 
$\Prop(\c S)$ to be a complete Heyting algebra \cite[Theorem 3.5]{gold:cove11}, and hence a model of intuitionistic logic. Since we are interested in non-distributive residuated lattices, any such constraint must be abandoned.

\section{Residuated Cover Systems}
To make $\Prop(\S)$ into a  residuated pomonoid we add a  a binary operation $\cd$ on $S$, which  will also be called \emph{fusion} (hopefully without causing confusion). This is lifted to a $\sub$-monotone binary operation on subsets of $S$ by putting, for $X,Y\sub S$,
$$
X\cd Y=\{x\cd y: x\in X\text{ and }y\in Y\}.
$$
We write $x\cdot Y$ for the set  $\{x\}\cd Y$, and $X\cd y$ for  $X\cd \{y\}$.

Define operations $\lr$ and $\rr$ on subsets of $S$  by
\begin{equation}  \label{defimp}
X\lr Y =\{z\in S: z\cd X\sub Y\}, \quad X\rr Y =\{z\in S: X\cd z\sub Y\}.
\end{equation}
These provide left and right residuals to the fusion operation on the complete poset $(\c P S,\sub)$, 
 i.e.\ for all $X,Y,Z\sub S$ we have
\begin{equation}   \label{residfu}
X\sub Y\lr Z  \iff  X\cd Y\sub Z \iff  Y\sub X\rr Z.
\end{equation}

Next define $X\circ Y$ to be the up-set $\up(X\cd Y)$ generated by $X\cd Y$. Then \emph{if $Z$ is an up-set}, we have 
$X\cd Y\sub Z$ iff  $X\circ Y\sub Z$, and hence \eqref{residfu} implies
\begin{equation}   \label{residcirc}
X\sub Y\lr Z  \iff  X\circ Y\sub Z \iff  Y\sub X\rr Z
\end{equation}
for any $X$ and $Y$. If the fusion operation $\cd$ is $\pre$-monotone in each argument, then $Y\lr Z $ and 
$X\rr Z$ are up-sets when $Z$ is an up-set. In particular, this implies that $\Up(\S)$ is closed under $\lr$ and $\rr$, and so by \eqref{residcirc}, these operations are  left and right residuals to $\circ$ on $\Up(\S)$. 

A \emph{residuated cover system} was defined in \citep{gold:grish11}  to be a structure of the form
$$
\S=(S,\pre,\covby,\cd,\e ),
$$
such that:
\begin{itemize}
\item 
$(S,\pre,\covby)$ is a cover system.
\item
$(S,\cd,\e )$ is a pomonoid, i.e.\ $\cd$ is an associative operation on $S$ that is $\pre$-monotone in each argument, and has $\e\in S$ as identity.
\item
\emph{Fusion preserves covering}: $x\covby C$ implies $x\cdot y\covby C\cdot y$ and $y\cdot x\covby y\cdot C$.
\item
\emph{Refinement of $\e$ is local}: $x\covby C\sub\up\e$ implies $\e\pre x$.
\end{itemize}
The last condition states that if $x$ locally refines $\e$, in the sense that it has a cover consisting of points refining $\e$, then $x$ itself refines $\e$. This means that the up-set $\up\e$ of points refining $\e$ is localised, i.e.\ $j\up\e\sub\up\e$, and therefore is a proposition. The condition that fusion preserves covering implies, more strongly, that
\begin{equation}  \label{covfu}
\text{if $x\covby C$ and $y\covby D$, then $x\cd y\covby C\cd D$.}
\end{equation} 
This was shown in     \citep{gold:grish11}, where the following was also established:

\begin{theorem}  \label{propS}
The set $\Prop(\S)$ of propositions of a residuated cover system $\S$ forms a complete residuated lattice under a monoidal operation $\fu$ with  identity\/  $1$, where
\begin{align*}
  X\fu Y       &= j(X\circ Y) = j\up (X\cd Y) \\
   1	&=\up\e \\
  {\textstyle \bigsqcap}\mathcal{X} &=  {\textstyle\bigcap}\mathcal{X} \\
   {\textstyle\bigsqcup}\mathcal{X} & =  j({\textstyle\bigcup}\mathcal{X})\\ 
  X\lr Y      & =\{z\in S:z\cd X\sub Y\} \\
  X\rr Y     &= \{z\in S:X\cd z\sub Y\} \\
   \top            &= S  \\
   \bot             &= j\emptyset  = \{x: x\covby\emptyset  \}. 
\end{align*}
\qed
\end{theorem}

\section{Modal FL-Cover Systems}

The linear logic semantics of \citet{gira:line87,gira:line95} uses the notion of a phase space, a certain kind of structure based on a commutative monoid $(M,\cd,\e)$. There is a closure system on the set of subsets of $M$, and formulas are interpreted as closed sets, which are called \emph{facts}. There is a monoidal structure on the set of facts in which $X\fu Y$ is the least fact including $X\cd Y$, and the identity $1$ of $\fu$ is the least fact containing $\e$.

The storage modality was interpreted in  \citep{gira:line95} as an operation $\sh$ on the algebra of facts that defines $\sh X$ to be the least fact including $X\cap I$, where $I$ is the submonoid of $M$ consisting of all elements $x$ of $1$ that are idempotent, i.e. $x\cd x=x$.
This approach was generalized by \citet{lafo:fini97}, replacing $I$ by the larger submonoid $J(M)$ of all elements $x\in 1$ such that $x$ belongs to the closure of $x\cd x$, and defining $\sh X$ to be the least fact including $X\cap K$, where $K$ is some designated submonoid of $J(M)$. That definition was then taken up  by \citet{okad:fini99}.

Here we will use a mix of these ideas to define storage modalities on the algebra  $\Prop(\S)$ of propositions (=localized up-sets) of a residuated cover system $\S$. In that context, the least proposition including a set $X$ is $j\up X$, and we will define $\sh X$ to be $j\up(X\cap I)$, where $I$ is a designated subset of $\up\e$ (=1) that is a cover of $\e$ and a submonoid of $(S,\cd,\e )$, as well as  consisting of idempotents that commute with all members of $S$. But whereas $\sh$ and $\qu$ are interdefinable in linear logic, here we treat them as independent, and model $\qu$ by the same interpretation that Kripke gave to the classical modality $\Diamond$. We take a binary relation $R$ on $S$ and let $\qu X$ be the set 
\begin{equation}\label{defdi}
\<R\>X=\{x\in S:\exists y(xRy\in X)\}.
\end{equation}
In \citep{gold:krip06,gold:cove11} we showed that any monotone operation on $\Prop(\S)$ can be given such a modelling, provided the cover system interacts in specified ways with the relation $R$.

\begin{definition}  \label{def3}
A \emph{modal FL-cover system} is a structure
$$
\S=(S,\pre,\covby,\cd,\e, 0,I,R ),
$$
with $0\in\Prop(\S)$, $I\sub \up\e$ and $R\sub S\times S$, such that:
\begin{itemize}
\item 
$(S,\pre,\covby,\cd,\e)$ is a residuated cover system.
\item
$I$ is a submonoid of $(S,\cd,\e)$, i.e.\ $I$ is closed under $\cd$ and contains $\e$.
\item
$I$ is an $\e$-cover:  $\e\covby I$.
\item
\emph{$I$ is idempotent and central}: if $x\in I$, then $x=x\cd x$ and $x\cd y=y\cd x$ for all $y\in S$.
\item 
\emph{$\pre$ and $R$ are confluent}: if $x\pre y$ and $xRz$, then there exists $w$ with $z\pre w$ and $yRw$;
\item
\emph{Modal Localisation}: if there exists an $x$-cover included in $\di{R}X$, then there exists a $y$ with $xRy$
and a $y$-cover included in $X$.
\item
\emph{$R$-Monotonicity}: If $x\in I$ and $yRz$, then $x\cd yRx\cd z$ and $y\cd xRz\cd x$.
\item
$R$ is reflexive and transitive.
\item
$xRy\in 0$ implies $x\in 0$.
\item$x\in 0$ implies that for some $y$, $xRy\covby\emp$.
\end{itemize}
\end{definition}
This definition looks formidable but has a certain inevitability. Its conditions are those that are needed to show that the lattice  of propositions of $\S$ is a complete modal FL-algebra. The details of how this works are given in the proof of the next theorem, but first we give a summary. We have already observed that $\Prop(\S)$ is a complete residuated lattice when $\S$ is based on a residuated cover system. The specified proposition 0 then serves as the distinguished element making $\Prop(\S)$ into an FL-algebra. The listed conditions on $I$ are used\footnote{Except for the condition $\e\in I$: see the  note at the end of this section.} to show that the operation $\sh X=j\up(X\cap I)$ is a storage modality on $\Prop(\S)$, i.e.\ satisfies (s1)--(s5) of Definition \ref{def1}. In particular, while (s1) and (s2) hold for any $I$, the proof of (s3) uses both $I\sub \up\e$ and
$\e\covby I$; that of (s4) uses that $I$ is  closed under $\cd$ and idempotent; and (s5) uses that $I$ is central.
The listed conditions on $R$ are used to show that the operation
 $\qu X=\<R\>X$ satisfies (c1)--(c5) of Definition \ref{def2}. The confluence of $\pre$ and $R$ and Modal Localisation together ensure that $\<R\>X$ belongs to $\Prop(S)$ whenever $X$ does; the proof of (c1) uses $R$-monotonicity; those of (c2) and (c3) use reflexivity and transitivity of $R$, respectively; and the last two conditions on $R$ are used for (c4) and (c5) respectively.

We turn now to the details.

\begin{theorem}
If $\S=(S,\pre,\covby,\cd,\e, 0,I,R )$ is a modal FL-cover system, then
$$
\L_\S=(\Prop(\S), \sub,\fu,1,\lr,\rr,0,\sh,\qu),
$$
is a complete modal FL-algebra, where the operations $\fu,1,\lr,\rr$ are as given in Theorem \ref{propS}; the storage modality is defined by $\sh X=j\up(X\cap I)$; and the consumption modality is given by $\qu X=\di{R}X$ as defined in \eqref{defdi}.
\end{theorem}

\begin{proof}
By Theorem 1, $(\Prop(\S), \sub,\fu,1,\lr,\rr,0)$ is a complete FL-algebra. That it is closed under $\sh$ follows because $j\up Y$ is a proposition for any $Y\sub S$, hence in particular $\sh X=j\up(X\cap I)\in\Prop(\S)$. Confluence of $\pre$ and $R$ ensures that $\di{R}X$ is an up-set if $X$ is, while Modal Localisation  ensures that $\di{R}X$ is localised if $X$ is
\cite[p1047]{gold:cove11}. Thus $Prop(\S)$ is closed under the modality $\qu$.

We verify the axioms of Definitions \ref{def1} and \ref{def2} for any propositions $X,Y\in\Prop(\S)$.
\begin{enumerate}
\item [(s1):]
If $x\in\sh X$, then for some $C$, $x\covby C\sub \up(X\cap I)\sub \up X =X$ (as $X$ is an up-set). So $x\covby C\sub X$, showing $x\in jX\sub X$ (as $X$ is localised).  This proves $\sh X\sub X$.
\item [(s2):]
Since $X\cap I\sub j\up(X\cap I)=\sh X$, we get
$j\up(X\cap I)\sub j\up(\sh X\cap I)$, i.e.\ $\sh X\sub\sh\sh X$.
\item [(s3):]
We want $\sh 1=1$, i.e.\ $j\up(\up\e\cap I)=\up\e$. Since $I\sub\up\e$, this simplifies to $j\up I=\up\e$.
Now from  $I\sub\up\e$ we get $j\up I\sub  j\up\up\e=j\up\e\sub \up\e$ (since $\up\e$ is a proposition).
For the converse inclusion, as $I$ is an $\e$-cover, $\e\covby I\sub\up I$, hence $\e\in j\up I$ and so $\up\e\sub j\up I$ as $j\up I$ is an up-set.
\item [(s4):]
We want  $\sh(X\cap Y)=\sh X\fu\sh Y$, i.e.\ $j\up(X\cap Y\cap I)= j\up(\sh X\cd\sh Y)$.

Now if $x\in \sh(X\cap Y)$, then  $x\covby C\sub \up(X\cap Y\cap I)$ for some $C$. If $c$ is any member of $C$, then there is some $d$ with $c\suc d\in  X\cap Y\cap I\sub\sh X\cap\sh Y$. But then $d\cd d\in \sh X\cd\sh Y$, so as $d\suc d\cd d$ by idempotence, this leads to $c\in \up(\sh X\cd\sh Y)$. Altogether this shows that
$x\covby C\sub\up(\sh X\cd\sh Y)$, hence $x\in j\up(\sh X\cd\sh Y)=\sh X\fu\sh Y$ as required.

For the converse inclusion it suffices to show that $\sh X\cd\sh Y\sub \sh(X\cap Y)$, since this forces 
$\sh X\fu\sh Y= j\up(\sh X\cd\sh Y)\sub  \sh(X\cap Y)$, because
$\sh(X\cap Y)$ is a proposition and $j\up(\sh X\cd\sh Y)$ is the least proposition incuding
$\sh X\cd\sh Y$. So suppose $x\cd y\in \sh X\cd\sh Y$, where $x\in \sh X$ and $y\in\sh Y$.
Then $x\covby C_x\sub \up(X\cap I)$ and $y\covby C_y\sub \up(Y\cap I)$ for some $C_x$ and $C_y$.
Hence $x\cd y\covby C_x\cd C_y$ by the strong form \eqref{covfu} of preservation of covering by fusion. Now take any element $c\cd c'$ of $C_x\cd C_y$.
Then there exist $d,d'$ with $c\suc d\in X\cap I$ and $c'\suc d'\in Y\cap I$. 
Thus $d\cd d'\in X\cd I\sub X\cd\up\e  \sub X\fu\up\e=X$.
Similarly $d\cd d'\in I\cd Y\sub Y$. Therefore
$c\cd c'\suc d\cd d'\in X\cap Y$. Moreover, $d\cd d'\in I$ as $I$ is closed under fusion. It follows that
$c\cd c'\in\up(X\cap Y\cap I)$.
Altogether we now have $x\cd y\covby C_x\cd C_y\sub\up(X\cap Y\cap I)$, so 
$x\cd y \in j\up(X\cap Y\cap I)=\sh(X\cap Y)$ as required.

\item [(s5):]
We want  $\sh X\fu Y=Y\fu\sh X$. First, to show $\sh X\fu Y\sub Y\fu\sh X$ it suffices to show that $\sh X\cd Y\sub Y\fu\sh X$, since $\sh X\fu Y$ is the least proposition including $\sh X\cd Y$. So take $x\in\sh X$ and $y\in Y$. Then $x\covby C\sub\up(X\cap I)$ for some $C$. Now for any $c\in C$, there is some $c'\in X\cap I$ with $c\suc c'$. Then
 $c\cd y\suc c'\cd y=y\cd c'$, with the last equality holding because $c'\in I$ and $I$ is central. But $y\cd c'\in Y\cd(X\cap I)\sub Y\cd\sh X $, so $c\cd y\in\up(Y\cd\sh X )$. This proves that  $C\cd y\sub \up(Y\cd\sh X )$. But $x\cd y\covby C\cd y$, as fusion preserves covering, so then $x\cd y\in j\up(Y\cd\sh X )=Y\fu\sh X$.
 
 That completes the proof that  $\sh X\cd Y\sub Y\fu\sh X$, and hence that $\sh X\fu Y\sub Y\fu\sh X$. The proof of the converse inclusion  $Y\fu\sh X \sub\sh X\fu Y $ is similar.
 
 \item [(c1):]
 We want  $\sh (X\ir Y) \sub\qu X\ir\qu Y$ for $i=l,r$. Taking the case of $i=l$, let $x\in \sh (X\lr Y)$, so that 
 $x\covby C\sub\up((X\lr Y)\cap I)$ for some $C$. Then $x\cd y\covby C\cd y$ as fusion preserves covering.
 To show that $x\in \qu X\lr\qu Y$, take any $y\in\qu X$: we must then prove that $x\cd y\in\qu Y$.
 We have $yRz\in X$ for some $z$. 
 For each $c\in C$ there exists $d\in (X\lr Y)\cap I$ with $c\suc d$, hence $c\cd y\suc d\cd y$. Since $d\in I$, by $R$-monotonicity this implies $d\cd yRd\cd z\in (X\lr Y)\cd X\sub Y$,  so $c\cd y\in \up\di{R}Y=\qu Y$.
 Altogether this shows that $x\cd y\covby C\cd y\sub\qu Y$, so $x\cd y\in j\qu Y=\qu Y$, as $\qu Y$  is a proposition.
 
 The case of $i=r$ is similar, using the facts that that $x\covby C$ implies $y\cd x\covby y\cd C$, and $yRz$ implies 
 $y\cd dRz\cd d$ when $d\in I$.
 
\item [(c2):]
That $X\sub\qu X$ follows because $R$ is reflexive.
  
 \item [(c3):]
That $\qu\qu X\sub\qu X$ follows because $R$ is transitive.

 \item [(c4):]
That $\qu 0\sub 0$ corresponds exactly to the $\S$-condition that $xRy\in 0$ implies $x\in 0$.

\item [(c5):]
To show $ 0\sub \qu X$, let $x\in 0$. Then for some $y$, $xRy\covby\emp\sub X$. Now $y\in jX= X$ as $X$ is a proposition, hence $x\in\di{R}X=\qu X$. 
\qed
\end{enumerate}
\end{proof}
It is notable that the only part of idempotence of $I$ that was used was the condition $d\suc d\cd d$ for $d\in I$ in the proof of (s4). The reverse inequality $d\pre d\cd d$  holds independently of idempotence, for if $d\in I$ then $\e\pre d$, so $d=d\cd\e\pre d\cd d$. Thus if $\pre$ is a partial order (i.e.also anti-symmetric) it is enough to require that $d\pre d\cd d$ in order to have $d=d\cd d$ for $d\in I$.

Note also that we made no of use the requirement that $I$ contains $\e$ in the above result. But we will see that any modal FL-algebra is representable as an algebra of propositions based on a cover system that does have  $\e\in I$ and $\pre$ anti-symmetric.

\section{Representation of Modal FL-Algebras}   \label{sec6}

Let $\L=(L,\sqleq,\fu,1,\lr,\rr)$ be an order-complete residuated lattice. Define a structure $\S^\L=(S,\pre,\covby,\cd,\e )$ by putting $S=L$; $x\pre y$ iff $y\sqleq x$; $x\covby C$ iff $x\sqleq\join C$; $x\cd y=x\fu y$; and $\e=1$.  Then $\S^\L$ is a residuated cover system, and moreover is one in which the monoid operation $\fu$ on $\Prop(\S^\L)$ has 
$X\fu Y=X\circ Y=\up(X\cd Y)$. A cover system will be called \emph{strong} if it satisfies this last condition $X\fu Y=X\circ Y$.

Note that in $\S^\L$, the up-set $\up x=\{y:x\pre y\}$ is equal to $\{y:y\sqleq x\}$, which is the \emph{down-set} of $x$ in $(L,\sqleq)$. In fact the propositions of $\S^\L$ are precisely these up-sets: if $X\in \Prop(\S^\L)$, then $X=\up x$ where 
$x=\join X$.  The map $x\mapsto \up x$ is order-invariant, in the sense that $x\sqleq y$ iff $\up x\sub\up y$, and is an isomorphism between the complete posets $(L,\sqleq)$ and $(\Prop(\S^\L),\sub)$, preserving all joins and meets. It also has 
$\up(x{\fu} y) = (\up x)\circ(\up y)$ and $\up(1)=\up\e$, and so is an isomorphism between $\L$ and the complete residuated lattice $\Prop(\S^\L)$ as described in Theorem \ref{propS}.

In this way we see that every order-complete residuated lattice  is isomorphic to the full algebra of all propositions of some residuation cover system. The proofs of these claims about the relationship between $\L$ and $\S^\L$ are set out in detail in \cite[Section 3]{gold:grish11}.

Now suppose $\L$ is a modal FL-algebra $(L,\sqleq,\fu,1,\lr,\rr,0,\sh,\qu)$. Expand the  above  residuated cover system $\S^\L$ to a structure $(S,\pre,\covby,\cd,\e, 0^\L,I,R )$ by adding the definitions
\begin{align*}
0^\L &= \up 0 = \{x:x\sqleq 0\},
\\
I &= \{\sh x:x\in L\},
\\
R &=\{(x,y): x\sqleq \qu y\}.
\end{align*}

\begin{lemma}
$\S^\L$ is a modal FL-cover system in which, for all $x\in L$, 
\begin{align}
\up(\sh x) &=j\up((\up x)\cap I), \quad \text{and}   \label{upsh}
\\
\up(\qu x) &=\di{R}(\up x).  \label{upqu}
\end{align}
\end{lemma}

\begin{proof}
We will apply the axioms for $\L$ given in Definitions \ref{def1} and \ref{def2} and the properties derived in Lemmas \ref{lem1} and \ref{lem2}. Observe that $0^\L$ belongs to $\Prop(\S^\L)$ because it is of the form $\up x$; and that $I\sub\up\e$ because in general $\sh x\sqleq 1=\e$, so $\e\pre\sh x$. We show that $\S^\L$ fulfills the list of conditions of Definition \ref{def3}.

\begin{itemize}
\item 
$I$ is a submonoid of $(S,\cd,\e)$:
the equation $\sh x\fu\sh y=\sh (\sh x\fu\sh y)$ ensures that $I$ is closed under $\cd$, while $\e=1=\sh1$ 
ensures that $\e\in I$.
\item
$I$ is an $\e$-cover:
 $\e\in I$ implies that $\e\sqleq\join I$, showing $\e\covby I$.
 \item
 $I$ is idempotent because $\sh x\fu\sh x=\sh x$; and central because $\sh x\fu y=y\fu\sh x$ (s5).
 \item
Confluence of$\pre$ and $R$: if $x\pre y$ and $xRz$, then $y\sqleq x\sqleq\qu z$, so $yRz$. Putting $w=z$ gives 
 $z\pre w$ and $yRw$.
\item
Modal Localisation: 
suppose there is a $C$ with $x\covby C\sub\di{R}X$. Then $x\sqleq\join  C$.
Let $C'=\{z\in X: \exists c\in C\,(cRz)\}$ and put $y=\bscup C'$. Then $C'$ is a $y$-cover included in $X$, and so it remains only to show that $xR y$. But if $c\in C$, then  by supposition there exists $z$ with $cRz\in X$. Then $c\sqleq \qu z$ and $z\in C'$, so $z\sqleq y$, hence $\qu z\sqleq \qu y$ as $\qu$ is monotone, and thus $c\sqleq \qu y$. Therefore $x\sqleq\join C\sqleq \qu y$, giving $xRy$ as required.
 \item
 $R$ is reflexive as $x\sqleq\qu x$ (c2) gives $xRx$. $R$ is transitive because if $xRyRz$, then $x\sqleq \qu y$ and
 $y\sqleq \qu z$, hence $\qu y\sqleq \qu\qu z\sqleq\qu z$ (c3), so $x\sqleq \qu z$, i.e.\ $xRz$.
 \item
 \emph{$R$-Monotonicity}: If $x\in I$ and $yRz$, then $x=\sh w$ for some $w$ and $y\sqleq \qu z$. Using (s2), 
 $\sqleq$-monotonicity of $\fu$ and then (c6), we reason that 
 $$
 x\cd y =\sh w\fu y\sqleq\sh\sh w\fu\qu z\sqleq\qu(\sh w\fu z)=\qu(x \cd z),
 $$
 so $x\cd yRx\cd z$. The proof that $y\cd xRz\cd x$ is similar, using (c7).
 \item
 If $xRy\in 0^\L$, then $x\sqleq \qu y$ and $y\sqleq 0$. Hence $\qu y\sqleq\qu 0\sqleq 0$ (c4), implying $x\sqleq 0$ and hence $x\in 0^\L$.
 \item
 Let $x\in 0^\L$. To show $\exists y:xRy\covby\emp$, put
  $y=\join\emp$. Then indeed $y\covby \emp$, and since $x\sqleq 0$ and $0\sqleq\qu y$  (c5), we have $xRy$ as required.
\end{itemize}
That completes the proof that  $\S^\L$ is a modal FL-cover system. To prove \eqref{upsh}, we first show
\begin{equation}\label{shjoin}
\sh x=\join\up(\up x\cap I).
\end{equation}
For, since $\sh x\sqleq x$ (s1) we have $\sh x\in\up x\cap I$, so $\sh x\sqleq \join\up(\up x\cap I)$.  But conversely, for any
 $y\in \up(\up x\cap I)$, there is some $\sh z\in \up x\cap I$ such that $y\suc\sh z$, hence $y\sqleq\sh z\sqleq x$.
 Then  $y\sqleq\sh z \sqleq\sh \sh z\sqleq\sh x$. Thus $\sh x$ is an upper bound of $\up(\up x\cap I)$, implying 
 $\join \up(\up x\cap I)\sqleq \sh x$ and proving \eqref{shjoin}.  Now for any $y\in S$ we reason that
 
$y\in j\up((\up x)\cap I)$

iff there is a $C$ with $y\covby C\sub \up((\up x)\cap I)$

iff there is a $C\sub  \up((\up x)\cap I)$ with $y\sqleq\join C$

iff $y\sqleq\join \up((\up x)\cap I)$

iff $y\sqleq\sh x$ by \eqref{shjoin}

iff $y\in\up(\sh x)$.

\noindent
That proves \eqref{upsh}. For \eqref{upqu}, if $y\in\up(\qu x)$, then $y\sqleq\qu x$ and so $yRx\in\up x$, showing $y\in \di{R}\up x$. Conversely, if $yRz\sqleq x$ for some $z$, then $y\sqleq\qu z\sqleq\qu x$, implying $y\in\up(\qu x)$.
\qed
\end{proof}
Results \eqref{upsh} and \eqref{upqu} state that the map $x\mapsto\up x$ preserves the modalities of the modal FL-algebras $\L$ and $\Prop(\S^\L)$. Thus we have altogether established

\begin{theorem}
Any  order-complete modal FL-algebra is isomorphic to  the modal FL-algebra of  all propositions of some modal FL-cover system that is strong, i.e.\  satisfies $X\fu Y=\up(X\cd Y)$.
\qed
\end{theorem}

Combining this with Ono's result, mentioned at the end of Section \ref{sec2}, on completions of modal FL-algebras, we have

\begin{theorem}
Any  modal FL-algebra is isomorphically embeddable into  the modal FL-algebra of  propositions of some strong modal FL-cover system, by an embedding that preserves all existing joins and meets.
\qed
\end{theorem}

\section{Kripke-Type Semantics}  \label{sec7}

\citet{ono:sema93} defined a modal substructural propositional logic that is sound and complete for validity in models on modal FL-algebras, as well as a first-order extension of this logic that is characterised by suitable models on modal FL-algebras. In view of the representation theorems just  obtained, these models can be taken to based on algebras of the form   $\Prop(\S)$, where $\S$ is a strong modal FL-cover system. In any such model $\M$, each sentence $\ph$ is interpreted as a proposition $\abm{\ph}\in\Prop(\S)$, with propositional constants and connectives interpreted by the operations of Theorem \ref{propS}; quantifiers interpreted using the join and meet operations $\join$ and $\meet$ of that Theorem; and modalities interpreted by $\sh X=j\up(X\cap I)$ and $\qu X=\di{R}X$. Writing $\M,x\models\ph$ to mean that $x\in\abm{\ph}$, and unravelling the definitions of the operations on $\Prop(\S)$, results in a Kripke-style satisfaction relation between formulas and points in models on cover systems. We now briefly present the inductive clauses specifying such a satisfaction relation in models for a certain type of first-order language.

Let $\LL$ be a signature, comprising a collection of individual constants $c$, and predicate symbols $P$ with specified arities $n<\omega$.  The \emph{$\LL$-terms}  are the individual constants $c\in\LL$ and the individual variables $v$ from some fixed denumerable list of such variables. An \emph{atomic $\LL$-formula} is any expression $P\tau_1\cdots \tau_n$ where $P\in\LL$ is $n$-ary, and the $\tau_i$ are $\LL$-terms.
The set of all \emph{$\LL$-formulas} is generated from the atomic $\LL$-formulas  and constant formulas $\top$, $\bot$,  
$\so$, $\sz$, using  the propositional connectives
$\land$ , $\lor$,  $\li$ and $\ri$ (interpreted as $\sqcap$, $\sqcup$, $\lr$ and $\rr$);   the quantifiers $\forall v$, $\exists v$ for all variables $v$; and the modalities $\sh$ and $\qu$.

An \emph{$\LL$-model}  $\M=(\S,U,\abm{-})$ has $\S=(S,\pre,\covby,\cd,\e, 0^\S,I,R )$  a modal FL-cover system, $U$  a non-empty set (universe of individuals), and $\abm{-}$  an interpretation function assigning
\begin{itemize}
\item
to each individual constant $c\in\LL$ an element $\abm{c}\in U$; \enspace and
\item 
to each $n$-ary predicate symbol $P\in\LL$, a function $\abm{P}:U^n\to\Prop(\S)$.
\end{itemize}
Intuitively, $\abm{P}({u_1},\dots,{u_n})$ is the proposition asserting that the predicate $P$ holds of the $n$-tuple of individuals $({u_1},\dots,{u_n})$. 

Let $\LL^U$ be the extension of $\LL$ to include the members of $U$ as individual constants. $\M$ automatically extends to an $\LL^U$-model by putting $\abm{c}=c$ for all $c\in U$. $\M$ has a truth/satisfaction relation $\M,x\models\ph$ between elements $x\in S$ and \emph{sentences} $\ph$ of $\LL^U$, with associated \emph{truth-sets}
$
\abm{\ph}=\{x\in S:\M,x\models\ph\}.
$
These notions are defined by induction on the length of $\ph$, as follows.

\begin{center}
\begin{tabular}{lc@{\quad }l}
$\M,x\models Pc_{1}\cdots c_{n}$ &iff & $x\in \abm{P}( \abm{c_1},\dots,\abm{c_n})$
\\
$\M,x\models \top$ &    &
\\
$\M,x\models \bot$ & iff   &$x\covby\emp$
\\
$\M,x\models \so$ & iff   &$\e\pre x$
\\
$\M,x\models \sz$ & iff   &$x\in 0^\S$
\\
$\M,x\models \ph\land\psi$ & iff &    $\M,x\models \ph$ and $ \M,x \models \psi$
\\
$\M,x\models\ph\lor\psi$     & iff & there is an $x$-cover $C\sub\abm{\ph}\cup\abm{\psi}$
\\
$\M,x\models \ph\li\psi$ & iff &  $ \M,y\models \ph$  implies $\M,x\cd y\models \psi$
\\
$\M,x\models \ph\ri\psi$ & iff &  $ \M,y\models \ph$  implies $\M,y\cd x\models \psi$
\\
$\M,x\models \forall v\ph$	 & iff & for all $c\in U$, $\M,x\models \ph(c/v) $ 
\\
$\M,x\models \exists v\ph$	 & iff & there is an $x$-cover $C\sub\bcup_{c\in U}\abm{\ph(c/v)}$
\\
$\M,x\models \sh\ph$ & iff   & there is an $x$-cover $C\sub\up(\abm{\ph}\cap I)$
\\
$\M,x\models \qu\ph$ & iff   & for some $y$, $xRy$ and $\M,y\models\ph$.
\end{tabular}
\end{center}
A sentence $\ph$ is   \emph{true in model $\M$} if it is true at every point, i.e.\  if $\M,x\models\ph$ for all $x\in S$, or equivalently $\abm{\ph}=S$.  A formula $\ph$ with free variables is \emph{true in} $\M$ if every $\LL^U$-\emph{sentence} $\ph(c_1/v_1,\dots,c_n/v_n)$ that is a substitution instance of $\ph$ is true in $\M$.

Detailed discussion of this kind of cover system semantics, and associated completeness theorems axiomatising their valid sentences, are presented in \citep{gold:krip06} for the logic of non-modal FL-algebras; in \citep{gold:cove11} for intuitionistic modal first-order logics; in \citep[Chapter 6]{gold:quan11} for propositional and quantified relevant logics; and in \citep{gold:grish11} for a `classical' version of bilinear logic that we also discuss below in Section \ref{secgrish}.

\section{Negation and Orthogonality}
Any FL-algebra has two unary `negation-like' operations,  $\lmin$ and $\rmin$, defined by putting $\lmin a=a\lr 0$ and 
$\rmin a=a\rr 0$. In the algebra $\Prop(\S)$ of a residuated cover system with a distinguished proposition 0, these operations can be analysed by the  `orthogonality' relation  $\perp$ on $S$ defined by
\begin{equation} \label{defperp}
z\perp y  \iff  z\cd y\in 0.
\end{equation}
Writing $z\perp X$ when $z\perp y$ for all $y\in X$,
and
$X\perp z$ when  $y\perp z$ for all $y\in X$, we get that $z\perp X$ iff $z\cd X\sub 0$ and $X\perp z$ iff $X\cd z\sub 0$, so

\begin{equation}  \label{perplrneg}
\lmin X =\{z\in S: z\perp X\} \quad \text{and}\quad \rmin X =\{z\in S: X\perp z\}.
\end{equation}

The operations $\lmin$ and $\rmin$  interpret left and right negation connectives, defined by taking $\lneg\ph$ to be $\ph\li\sz$ and $\rneg\ph$ to be $\ph\ri\sz$. These have the semantics

\begin{align} \label{semneg}
\begin{split}
&\M,x\models \lneg\ph \quad\text{iff}\quad x\perp\abm{\ph}
\\
&\M,x\models \rneg\ph \quad\text{iff}\quad \abm{\ph}\perp x.
\end{split}
\end{align}

When $\cd$ is commutative, the relation $\perp$ is symmetric and $\lmin$ and $\rmin$ are identical. But even in the absence of symmetry we do have $z\perp\e$ iff $\e\perp z$ iff $z\in 0$. So $0$ itself is recoverable from $\perp$ as the set 
$\{z:z\perp\e\}=\{z:\e\perp z\}$.

Now if $y\in\up\e$, then in general $z\cd\e\pre z\cd y$, so $z\perp\e$ implies $z\perp y$ as 0 is an up-set. This shows that
$z\perp\e$ iff $z\perp\up\e$ iff $z\in\lmin\up\e$. Similarly, $\e\perp z$ iff $\up\e\perp z$ iff $z\in\rmin\up\e$. Thus in $\Prop(\S)$ we have
$
0=\lmin 1=\rmin 1.
$
When $\S$ is a modal FL-cover system, we also have
\begin{equation}
0=\lmin I=\rmin I.
\end{equation}
To see why, note that since $I\sub\up\e$, we have $0=\lmin\up\e\sub\lmin I$. But if $z\in\lmin I$, then $z\perp I$, so $z\perp\e$ as $\e\in I$, hence $z\in 0$. This shows that $0=\lmin I$. The proof that $0=\rmin I$ is similar.

The relation $\perp$ defined in \eqref{defperp} has the following properties:
\begin{itemize}
\item 
$z\perp y$    iff   $z\cd y\perp\e$.
\item
\emph{Orthogonality to $\e$ is monotonic}:  $y\suc z\perp\e$ implies $y\perp\e$.
\item
\emph{Orthogonality to $\e$ is local}:  $x\covby C\perp\e$ implies $x\perp\e$.
\end{itemize}
Vice versa, if we begin with a residuated cover system $\S$ having a binary relation $\perp$ with these properties, then it follows that $\{z:z\perp\e\}=\{z:\e\perp z\}$, and that this set belongs to $\Prop(\S)$. So we can take it as the \emph{definition} of 0. Then the sets $\lmin X$ and $\rmin X$ defined from $\perp$ as in \eqref{perplrneg} turn out to be $X\lr 0$ and $X\rr 0$ for this choice of 0, respectively.

The modelling of negation by an orthogonality relation as in \eqref{semneg} first occurred in \citep{gold:sema74}, with $\perp$ symmetric. The idea of defining $\perp$ from a distinguished subset of a monoid as in \eqref{defperp} is due to \citet{gira:line87}.

We will make use of  some basic properties of $\lmin$ and $\rmin$ in any FL-algebra (see e.g.\ \citep[Section 2.2]{gala:resi07}):
\begin{itemize}
\item 
$a\sqleq\lmin\rmin a$ and $a\sqleq\rmin\lmin a$.
\item
$a\sqleq b$ implies $\imin b\sqleq\imin a$ for $i=l,r$ \quad(antitonicity).
\item
$\lmin 1=0=\rmin 1$.
\item
$1\sqleq\lmin 0\sqcap\rmin 0$.
\item
$a\rr b\sqleq  \rmin b\lr\rmin a$.
\item
$a\lr b\sqleq  \lmin b\rr\lmin a$.
\end{itemize}
For instance, the last `contrapositive' inequality is the  case $c=0$ of
\begin{equation}\label{contra}
a\lr b\sqleq   (b\lr c)\rr  (a\lr c),
\end{equation}
which is itself shown by using residuation to reason  that
$$
(b\lr c)\fu(a\lr b)\fu a\sqleq (b\lr c)\fu b\sqleq c,
$$
implying that $(b\lr c)\fu(a\lr b) \sqleq (a\lr c)$, from which \eqref{contra} follows.

Now in Boolean modal algebra, a modality $\sh$ has the dual modality  $\minus\sh \minus$, where  $\minus$ is the Boolean complement/negation operation. Given the two negations $\lmin$ and $\rmin$ we would seem to have four possiblities here for defining a term function to which $\qu$ is  dual. But it turns out that they are all the same:

\begin{lemma} 
Any FL-algebra with a storage modality $\sh$ satisfies $\lmin\sh a=\rmin\sh a$ and  $\sh\lmin a=\sh\rmin a$, for all $a$. Hence
$$
\lmin\sh\rmin a=\rmin\sh\lmin a=\lmin\sh\lmin a=\rmin\sh\rmin a.
$$
\end{lemma}

\begin{proof}
By (s5) and then residuation, $(\sh a\rr 0)\fu\sh a=\sh a\fu (\sh a\rr 0)\sqleq 0$, implying that
$(\sh a\rr 0)\sqleq (\sh a\lr 0)$. Similarly $\sh a\fu (\sh a\lr 0)=(\sh a\lr 0)\fu\sh a\sqleq 0$, implying
$(\sh a\lr 0)\sqleq (\sh a\rr 0)$. Hence $(\sh a\lr 0)= (\sh a\rr 0)$, i.e.\  $\lmin\sh a=\rmin\sh a$.

Next, by (s5) and (s1),  $a\fu\sh(a\lr 0)=\sh(a\lr 0)\fu a\sqleq (a\lr 0)\fu a\sqleq 0$, so
$\sh(a\lr 0)\sqleq a\rr 0$. Hence by $\sh$-monotonicity and (s2), $\sh(a\lr 0)\sqleq \sh(a\rr 0)$.
The reverse inequality holds similarly, so $\sh(a\lr 0)= \sh(a\rr 0)$, i.e.\ $\sh\lmin a=\sh\rmin a$. 

The second statement of the Lemma follows from the first.
\qed
\end{proof}

Thus we can define an operation $\qus$ by writing $\qus a$ for the element $\lmin\sh\rmin a$, or any of its  three other manifestations as given by this last result.

\begin{theorem}
If\/ $\sh$ is a storage modality on an FL-algebra $\L$, then the operation $\qus$ satisfies  the axioms (c1)--(c5) and so, together with $\sh$, makes $\L$ into a modal FL-algebra.
\end{theorem}

\begin{proof}\strut
\begin{enumerate}[\rm(c1)]
\item
From the contrapositve inequality $a\lr b\sqleq  \lmin b\rr\lmin a$,  by $\sh$-monotonicity and Lemma \ref{lem1}(5) we get
$\sh(a\lr b)\sqleq  \sh\lmin b\rr\sh\lmin a$. But
$
  \sh\lmin b\rr\sh\lmin a\sqleq \rmin\sh\lmin a \lr \rmin  \sh\lmin b = \qus a\lr\qus b,
  $
so  $\sh(a\lr b)\sqleq   \qus a\lr\qus b$.

  Similarly we show that
$\sh(a\rr b)\sqleq   \lmin\sh\rmin a \rr \lmin  \sh\rmin b = \qus a\rr\qus b$.
\item 
Since by (s1) $\sh(a\rr 0)\sqleq a\rr 0$, residuation gives $a\sqleq \sh(a\rr 0)\lr 0= \lmin\sh\rmin a=\qus a$.
\item
We have $\sh\rmin a\sqleq \rmin\lmin \sh\rmin a$, as an instance of $b\sqleq\rmin\lmin b$.
Hence  by $\sh$-monotonicity and (s2), $\sh\rmin a\sqleq \sh\rmin\lmin \sh\rmin a$. This together with antitonicity gives
$\lmin\sh\rmin\lmin \sh\rmin a \sqleq\lmin \sh\rmin a$, which says $\qus\qus a\sqleq\qus a$.
\item
$1=\sh 1\sqleq\sh\rmin 0$, hence $\lmin \sh\rmin 0 \sqleq\lmin 1=0$. This says $\qus 0\sqleq 0$.
\item
By Lemma \ref{lem1}(1), we have $\sh\rmin a\sqleq 1$. Hence $0=\lmin 1\sqleq \lmin \sh\rmin a =\qus a$.
\qed
\end{enumerate}

\end{proof}

\section{Classical/ Grishin Algebras} \label{secgrish}

\citet{ono:sema93} defined an FL-algebra to be \emph{classical} if it satisfies the equations
$$
(a\rr0)\lr 0 = a = (a\lr0)\rr 0.
$$
This can be written as   $\lmin\rmin a =a=\rmin\lmin a$, and will be called the law of \emph{double-negation elimination}.
Girard's linear logic is modelled by classical FL-algebras in which the fusion operation $\fu$ is commutative.

\citet{lamb:some95} defined a \emph{Grishin algebra} to be a lattice-ordered pomonoid that has two unary operations $\lmin$ and $\rmin$ and a distinguished element 0 that satisfies double-negation elimination and the conditions
$$
a\sqleq b \iff a\fu\rmin b\sqleq 0 \iff \lmin b\fu a\sqleq 0.
$$
He described such algebras as being ``a generalisation of Boolean algebras which do not obey Gentzen's three structural rules''. His motivation was to study algebraic models for \emph{classical bilinear propositional logic}, described as ``a non-commutative version of linear logic which allows two negations''. Such models were first considered by 
\cite{grish:gene83}.

 Lambek showed that a Grishin algebra can be equivalently defined as a residuated lattice with two operations 
$\lmin$ and $\rmin$ satisfying double-negation elimination and
$$
\lmin 1=\rmin 1, \quad a\lr b=\lmin(a\fu\rmin  b), \quad a\rr b=\rmin(\lmin b\fu a).   
$$
A proof that the notions of classical FL-algebra and Grishin algebra are equivalent is given in \cite[Theorem 2.2]{gold:grish11}.

A residuated cover system $\S$ will be called \emph{classical} if it has a distinguished proposition (localised up-set) 0 such that
the least proposition  containing any given $X$ is equal to both $\lmin\rmin X$ and $\rmin\lmin X$.
In other words,
\begin{equation} \label{classical}
j\up X = \lmin\rmin X=\rmin\lmin X
\end{equation}
holds for all $X\sub S$, where $\lmin$ and $\rmin$ are defined from $\lr$ and $\rr$ using 0. This is equivalent to requiring that $\Prop(\S)$ be a Grishin algebra/classical FL-algebra, and is also equivalent to the requirement that \eqref{classical} holds just  for all up-sets $X$ \citep[Theorem 4.2]{gold:grish11}.

We showed in \citep{gold:grish11} that every Grishin algebra has an isomorphic embedding into the algebra of all propositions of some strong classical residuated cover system, by a map that preserves all existing joins and meets. The method, involving MacNeille completion, can be combined with the constructions of this paper to give a representation of any classical modal FL-algebra as an algebra of propositions of some strong classical modal FL-cover system.

In conclusion we relate our constructions back to the modelling of consumption modalities in \citep{gira:line95}, which is  based on the notion of a phase space  as a commutative monoid with a distinguished subset (but without a preorder). Suppose $\S$ is a \emph{classical} modal FL-cover system in which $\cd$ is commutative. Then the relation $\perp$ defined in \eqref{defperp} is symmetric, and so  the sets $\lmin X$ and $\rmin X$ in \eqref{perplrneg} are one and the same. We denote this set by $X^\perp$. The operation $X\mapsto X^{\perp\perp}$ is a closure operator on the powerset of $S$ that has $X^{\perp\perp\perp}=X^\perp$. Moreover, $j\up X=X^{\perp\perp}$ according to \eqref{classical}. Now the modality $\qus$ on $\Prop(\S)$ is given by
$$
\qus X= (\sh X^\perp)^\perp = (j\up (X^\perp\cap I))^\perp =  (X^\perp\cap I)^{\perp\perp\perp}
=(X^\perp\cap I)^\perp.
$$
$(X^\perp\cap I)^\perp$ is  Girard's  definition of $\qu X$ when $I$ is the set of idempotents belonging to $1=\{\e\}^{\perp\perp}$.


\end{document}